\documentclass[final,12pt,3p,times]{elsarticle}
\usepackage{amssymb}
\usepackage{latexsym}
\usepackage{amssymb}
\usepackage{here}
\usepackage{wrapfig}
\usepackage{fancybox}
\usepackage{ascmac}
\usepackage{amsmath}
\usepackage{multicol}
\usepackage{mathtools}
\usepackage{amsthm}
\newtheorem{theo}{Theorem}[section]
\newtheorem{lem}[theo]{Lemma}
\newtheorem{rem}[theo]{Remark}

\newtheorem{theoA}{Theorem}[section]

\newtheorem{coroA}[theoA]{Corollary}

\DeclareMathOperator*{\esssup}{ess\,sup}

\allowdisplaybreaks[4]
\journal{Journal of Computational and Applied Mathematics}
\begin{document}
\begin{frontmatter}



\title{Sharp numerical inclusion of the best constant for embedding $H_{0}^{1}(\Omega) \hookrightarrow L^{p}(\Omega)$ on bounded convex domain}


\author[a]{Kazuaki Tanaka\corref{correspo}}
\cortext[correspo]{Corresponding author.}
\ead{imahazimari@fuji.waseda.jp}

\author[b]{Kouta Sekine}
\author[a]{Makoto Mizuguchi}
\author[b,c]{Shin'ichi Oishi}

\address[a]{Graduate School of Fundamental Science and Engineering, Waseda University, 3-4-1 Okubo, Shinjuku, Tokyo 169-8555, Japan}
\address[b]{Faculty of Science and Engineering, Waseda University, 3-4-1 Okubo, Shinjuku, Tokyo 169-8555, Japan}
\address[c]{CREST, JST, Japan}

\begin{abstract}
In this paper, we propose a verified numerical method for obtaining a sharp inclusion of the best constant for the embedding $H_{0}^{1}(\Omega) \hookrightarrow L^{p}(\Omega)$ on bounded convex domain in $\mathbb{R}^{2}$.
We estimate the best constant by computing the corresponding extremal function using a verified numerical computation.
Verified numerical inclusions of the best constant on a square domain are presented.
\end{abstract}


\end{frontmatter}

\section{Introduction}\label{sec/intro}
We consider the best constant for the embedding $H_{0}^{1}(\Omega) \hookrightarrow L^{p}(\Omega)$, i.e., the smallest constant $C_{p}\left(\Omega\right)$ that satisfies
\begin{align}
\left\|u\right\|_{L^{p}(\Omega)}\leq C_{p}(\Omega)\left\|u\right\|_{H_{0}^{1}\left(\Omega\right)},~~\forall u\in H_{0}^{1}\left(\Omega\right),\label{embedding}
\end{align}
where $\Omega\subset \mathbb{R}^{n}~(n=2,3,\cdots)$, $2< p<\infty$ if $n=2$, and $2< p\leq 2n/(n-2)$ if $n\geq 3$.
Here, $L^{p}\left(\Omega\right)$~$(1\leq p<\infty)$ is the functional space of $p$-th power Lebesgue integrable functions over $\Omega$.
Moreover, assuming that $H^{1}\left(\Omega\right)$ denotes first order $L^{2}$-Sobolev space on $\Omega$,
we define $H_{0}^{1}\left(\Omega\right):=\{u\in H^{1}\left(\Omega\right)\ :\ u=0\ \mathrm{on}\ \partial\Omega$ in the trace sense$\}$ with inner product $\left(\cdot,\cdot\right)_{H_{0}^{1}\left(\Omega\right)}:=\left(\nabla\cdot,\nabla\cdot\right)_{L^{2}\left(\Omega\right)}$ and norm $\left\|\cdot\right\|_{H_{0}^{1}\left(\Omega\right)}:=\left\|\nabla\cdot\right\|_{L^{2}\left(\Omega\right)}$.

Such constants are important in studies on partial differential equations (PDEs).	
In particular, our interest is in the applicability of these constants to verified numerical computation methods for PDEs, which originate from Nakao's \cite{nakao1988numerical} and Plum's work \cite{plum1991computer} and have been further developed by many researchers.
Such methods require explicit bounds for the embedding constant corresponding to a target equation at various points within them (see, e.g., \cite{nakao2001numerical, nakao2011numerical, plum2001computer, plum2008, takayasu2014remarks}).
Moreover, the precision in evaluating the embedding constants directly affects the precision of the verification results for the target equation.
Occasionally, rough estimates of the embedding constants lead to failure in the verification.
Therefore, accurately estimating such embedding constants is essential.

It is well known that the best constant in the classical Sobolev inequality has been proposed \cite{aubin1976, talenti1976} (see Theorem \ref{talentitheo}).
A rough upper bound of $C_{p}(\Omega)$ for a bounded domain $\Omega\subset \mathbb{R}^{n}$ can be obtained from the best constant by considering zero extension outside $\Omega$ (see Corollary \ref{roughboundtheo}).
Moreover, Plum \cite{plum2008} proposed another estimation formula that requires not the boundedness of $\Omega$ but an explicit lower bound for the minimum eigenvalue of $-\Delta$ (see Theorem \ref{plumembedding}), where $\Delta$ denotes the usual Laplace operator.
Although these formulas enable us to easily compute the upper bound of $C_{p}(\Omega)$, little is known about the best constant.

In this paper, we propose a numerical method for obtaining a verified sharp inclusion of the best constant $C_{p}\left(\Omega\right)$ that satisfies (\ref{embedding}) for a bounded convex domain $\Omega\subset \mathbb{R}^{2}$.
As a verified result, we prove the following theorem by using our method through a computer-assisted technique:

\begin{theo}\label{prop/embedding}
For the square $\Omega_{s}=\left(0,1\right)^{2}$, the smallest values of $C_{p}\left(\Omega_{s}\right)$~$(p=3,4,5,6,7)$ that satisfy $(\ref{embedding})$ are enclosed as follows:
\begin{align*}
C_{3}\left(\Omega_{s}\right)&\in[0.25712475017618,~0.25712475017620];\\
C_{4}\left(\Omega_{s}\right)&\in[0.28524446071925,~0.28524446071929];\\
C_{5}\left(\Omega_{s}\right)&\in[0.31058015094505,~0.31058015094512];\\
C_{6}\left(\Omega_{s}\right)&\in[0.33384042151102,~0.33384042151112];\\
C_{7}\left(\Omega_{s}\right)&\in[0.35547994288611,~0.35547994288634].
\end{align*}
\end{theo}
\begin{rem}\label{remark1}
Since it follows from a simple variable transformation that
\begin{align}
C_{p}((a,b)^{2})=(b-a)^{\frac{2}{p}}C_{p}(\Omega_{s}),
\end{align}
the values in Theorem \ref{prop/embedding} can be directly used for all squares $(a,b)^{2}$ $(-\infty<a<b<\infty)$ by multiplying them with $(b-a)^{2/p}$.
Moreover, these values can be applied to deriving an explicit upper bound of $C_{p}\left(\Omega\right)$ for a general domain $\Omega \subset (a,b)^{2}$ by considering zero extension outside $\Omega$, while the precision of the upper bound depends on the shape of $\Omega$.
\end{rem}
Hereafter, we replace the notation $C_{p}\left(\Omega\right)$ with $C_{p+1}\left(\Omega\right)$ ($1< p<\infty$ if $n=2$, and $1< p\leq (n+2)/(n-2)$ if $n\geq 3$) for the sake of convenience.
The smallest value of $C_{p+1}\left(\Omega\right)$ can be written as
\begin{align}
C_{p+1}\displaystyle \left(\Omega\right)=\sup_{u\in H_{0}^{1}\left(\Omega\right)\backslash \{0\}}\Phi\left(u\right),\label{bestconstant}
\end{align}
where $\Phi\left(u\right)=\left\|u\right\|_{L^{p+1}(\Omega)}/\left\|u\right\|_{H_{0}^{1}\left(\Omega\right)}$.
This variational problem is still the topic of current research (see, e.g., \cite{carroll2011interpolating, juhnke2015numerical} and the references therein).

The boundedness of $C_{p+1}\left(\Omega\right)$ in (\ref{bestconstant}) is ensured by considering zero extension outside $\Omega$ (see Corollary \ref{roughboundtheo}).
In addition, it is true that the supremum $C_{p+1}\left(\Omega\right)$ in (\ref{bestconstant}) can be realized by an extremal function in $H_{0}^{1}\left(\Omega\right)$.
A proof of this fact is sketched as follows.
Let $\{u_{i}\} \in H_{0}^{1}\left(\Omega\right)$ be a sequence such that $\left\|u_i\right\|_{H_{0}^{1}\left(\Omega\right)}=1$ and $\left\|u_{i}\right\|_{L^{p+1}(\Omega)}\rightarrow C_{p+1}(\Omega)$ as $i\rightarrow\infty$.
The Rellich--Kondrachov compactness theorem (see, e.g., \cite[Theorem 7.22] {gilbarg2001elliptic}) ensures that there exists a subsequence $\{u_{i_j}\}$ that converges to some $u^{*}$ in $L^{p+1}(\Omega)$.
Moreover, there exists a subsequence $\{u_{i{_k}}\} \subset \{u_{i{_j}}\}$ that converges to some $u' \in H^1_0(\Omega)$ in the weak topology of $H^1_0(\Omega)$ because $H^{1}_{0}(\Omega)$ is a Hilbert space.
Since $\{u_{i{_k}}\}$ converges to $u^{*}$ in $L^{p+1}(\Omega)$, it follows that $u^{*}=u'$.
Hence, $u^{*} \in H^1_0(\Omega) (\subset L_{p+1}(\Omega))$ and $\left\|u^{*}\right\|_{L^{p+1}(\Omega)}=C_{p+1}(\Omega)$.

Since $|u|\in H^{1}(\Omega)$ for all $u\in H^{1}(\Omega)$ (see, e.g.,  \cite[Lemma 7.6]{gilbarg2001elliptic}) and $\Phi(u^{*})=\Phi(|u^{*}|)$,
we are looking for the extremal function $u^{*}$ such that $u^{*}\geq0$ (in fact, the later discussion additionally proves that $u^{*}>0$ in $\Omega$).
The Euler-Lagrange equation for the variational problem is 
\begin{align}
\left\{\begin{array}{l l}
-\Delta u=lu^{p} &\mathrm{in}\ \Omega,\\
u=0 &\mathrm{on}\ \partial\Omega\\
\end{array}\right.\label{euler-lagrange}
\end{align}
with some positive constant $l$ (see, e.g., \cite{carroll2011interpolating} for a detailed proof).
Since $\Phi$ is scale-invariant (i.e., $\Phi(ku^{*})=\Phi(u^{*})$ for any $k>0$), it suffices to consider the case that $l=1$ for finding an extremal function $u^{*}$ of $\Phi$ (recall that we consider the case that $p>1$).
Moreover, the strong maximum principle ensures that nontrivial solutions $u$ to \eqref{euler-lagrange} such that $u \geq0$ in $\Omega$ are positive in $\Omega$.
Therefore, in order to find the extremal function $u^{*}$, we consider the problem of finding weak solutions to the following problem:
\begin{align}
\left\{\begin{array}{l l}
-\Delta u=u^{p} &\mathrm{in}\ \Omega,\\
u>0 &\mathrm{in}\ \Omega,\\
u=0 &\mathrm{on}\ \partial\Omega.\\
\end{array}\right.\label{positive/problem}
\end{align}
This problem has a unique solution if $\Omega\subset \mathbb{R}^{2}$ is bounded and convex \cite{lin1994uniqueness}.
Therefore, we can obtain an inclusion of $C_{p+1}\left(\Omega\right)$ as $\left\|u^{*}\right\|_{L^{p+1}(\Omega)}/\left\|u^{*}\right\|_{H_{0}^{1}\left(\Omega\right)}$ by enclosing the solution $u^{*}$ to \eqref{positive/problem} with verification.

Numerous numerical methods for verifying a solution to semilinear elliptic boundary value problems exist (e.g., \cite{nakao2001numerical, nakao2011numerical, plum2001computer, plum2008, takayasu2014remarks} along with related works \cite{nakao2005numerical,tanaka2014verified}).
Such methods enable a concrete ball containing exact solutions to elliptic equations to be obtained; this is typically in the sense of the norms $\left\|\cdot\right\|_{H_{0}^{1}\left(\Omega\right)}$ and $\left\|\cdot\right\|_{L^{\infty}\left(\Omega\right)}$, where $L^{\infty}\left(\Omega\right)$ is the functional space of Lebesgue measurable functions over $\Omega$ with the norm $\left\|u\right\|_{L^{\infty}\left(\Omega\right)}:=\esssup\{\left|u\left(x\right)\right|\,|\,x\in\Omega\}$ for $u\in L^{\infty}\left(\Omega\right)$.
In order to verify a solution to \eqref{positive/problem}, we first verify a solution to the following problem:
\begin{align}
\left\{\begin{array}{l l}
-\Delta u=\left|u\right|^{p-1}u &\mathrm{in}\ \Omega,\\
u=0 &\mathrm{on}\ \partial\Omega,\\
\end{array}\right.\label{abs/problem}
\end{align}
by combining the methods described in \cite{plum2001computer} and \cite{tanaka2014verified}.
Then, we prove the positiveness of the verified solution.
In exceptional cases, when $p$ is even, we directly consider problem \eqref{positive/problem}.
Indeed, the maximum principle ensures the positiveness of the solutions to the problem $-\Delta u=u^{p}~\mathrm{in}~\Omega,~u=0~\mathrm{on}~\partial\Omega$ a priori (note that $-\Delta u\geq 0$ when $p$ is even), except for the trivial solution $u\equiv 0$ in $\Omega$.


The remainder of this paper is organized as follows:
In Sections \ref{sec/positiveness} and \ref{sec/embedding}, we propose methods for proving the positiveness of the solution to (\ref{abs/problem}) and estimating the embedding constant $C_{p+1}\left(\Omega\right)$, respectively.
In Section \ref{sec/ex}, we present some numerical examples that yield Theorem \ref{prop/embedding}.

\section{Method for verifying positiveness}\label{sec/positiveness}
In this section, we propose a sufficient condition for the positiveness of the solution to (\ref{abs/problem}), which will be summarized in Theorem \ref{positive/theo}.
This theorem allows us to check for positiveness numerically.
Throughout this paper, for simplicity, we omit the expression ``almost everywhere'' for Lebesgue measurable functions. For example, we employ the notation $u>0$ in place of $u(x)>0$ a.e. $ x\in\Omega$.
We introduce the following lemma, which is required to prove Theorem \ref{positive/theo}.
\begin{lem}\label{main/lem}
Let $\Omega$ be a bounded domain in $\mathbb{R}^{n}$ $(n=1,2,3,\cdots)$.
All nontrivial weak solutions $0\not\equiv u\in H_{0}^{1}\left(\Omega\right)$ to $(\ref{positive/problem})$ satisfy
$$\esssup\{u\left(x\right)^{p-1}\,|\,x\in\Omega\}\geq\lambda_{1},$$
where $\lambda_{1}>0$ is the first eigenvalue of the following problem:
\begin{align}
\left(\nabla u,\nabla v\right)_{L^{2}\left(\Omega\right)}=\lambda\left(u,v\right)_{L^{2}\left(\Omega\right)},~~\forall v\in H_{0}^{1}\left(\Omega\right).\label{weak/eig/pro}
\end{align}
\end{lem}
\begin{proof}
Let $\phi_{1}\geq 0~(\phi_{1}\not\equiv 0)$ be the first eigenfunction corresponding to $\lambda_{1}$, which satisfies
\begin{align*}
\displaystyle \int_{\Omega}u^{p}\left(x\right)\phi_{1}\left(x\right)dx=\lambda_{1}\int_{\Omega}u\left(x\right)\phi_{1}\left(x\right)dx.
\end{align*}
We have
\begin{align*}
\displaystyle \int_{\Omega}u^{p}\left(x\right)\phi_{1}\left(x\right)dx&=\displaystyle \int_{\Omega}\left\{u\left(x\right)\right\}^{p-1}\left\{u\left(x\right)\phi_{1}\left(x\right)\right\}dx\\
&\displaystyle \leq \esssup\{u\left(x\right)^{p-1}\,|\,x\in\Omega\}\int_{\Omega}u\left(x\right)\phi_{1}\left(x\right)dx\\
&=\displaystyle \lambda_{1}^{-1}\esssup\{u\left(x\right)^{p-1}\,|\,x\in\Omega\}\int_{\Omega}u^{p}\left(x\right)\phi_{1}\left(x\right)dx,
\end{align*}
The positiveness of $\displaystyle \int_{\Omega}u^{p}\left(x\right)\phi_{1}\left(x\right)dx$ implies $\esssup\{u\left(x\right)^{p-1}\,|\,x\in\Omega\}\geq\lambda_{1}$.
\end{proof}

Using lemma \ref{main/lem}, we can prove the following theorem, which provides a sufficient condition for the positiveness of solutions to (\ref{abs/problem}).
\begin{theo}\label{positive/theo}
Let $\Omega$ be a bounded domain in $\mathbb{R}^{n}$ $(n=1,2,3,\cdots)$.
If a solution $u\in C^{2}\left(\Omega\right)\cap C\left(\overline{\Omega}\right)$ to $(\ref{abs/problem})$ is positive in a nonempty subdomain $\Omega'\subset\Omega$ and $\sup\{u_{-}\left(x\right)^{p-1}\,|\,x\in\Omega\}<\lambda_{1}\left(\Omega\right)$, then $u>0$ in the original domain $\Omega$; that is, $u$ is also a solution to $(\ref{positive/problem})$.
Here, $\lambda_{1}\left(\Omega\right)>0$ is the first eigenvalue of the problem $(\ref{weak/eig/pro})$ and $u_{-}$ is defined by
\begin{align*}
u_{-}\left(x\right):=
\left\{\begin{array}{l l}
-u\left(x\right), &u\left(x\right)<0,\\
0, &u\left(x\right)\geq 0.
\end{array}\right.
\end{align*}
\end{theo}
\begin{proof}
Assume that $u$ is not always positive in $\Omega$.
The strong maximum principle ensures that $u$ is also not always nonnegative in $\Omega$.
In other words, there exists a nonempty subdomain $\Omega''\subset\Omega\backslash\Omega'$ such that $u<0$ in $\Omega''$ and $u=0$ on $\partial\Omega''$.
Therefore, the restricted function $v:=-u|_{\Omega''}$ can be regarded as a solution to
\begin{align*}
\left\{\begin{array}{l l}
-\Delta v=v^{p} &\mathrm{i}\mathrm{n}\ \Omega'',\\
v>0 &\mathrm{i}\mathrm{n}\ \Omega'',\\
v=0 &\mathrm{on}\ \partial\Omega''.\\
\end{array}\right.
\end{align*}
From Lemma \ref{main/lem}, we have that
\begin{align*}
\sup_{x\in\Omega}u_{-}\left(x\right)^{p-1}&\geq\sup_{x\in\Omega''}v\left(x\right)^{p-1}
\geq\lambda_{1}\left(\Omega''\right),
\end{align*}
where $\lambda_{1}\left(\Omega''\right)$ is the first eigenvalue of $(\ref{weak/eig/pro})$, with the notational replacement $\Omega=\Omega''$.
Since the inclusion $\Omega''\subset\Omega$ ensures that all functions in $H_{0}^{1}\left(\Omega''\right)$ can be regarded as functions in $H_{0}^{1}\left(\Omega\right)$ by considering the zero extension outside $\Omega''$, the inequality $\lambda_{1}\left(\Omega''\right)\geq\lambda_{1}\left(\Omega\right)$ follows.
This contradicts the assumption $\sup\{u_{-}\left(x\right)^{p-1}\,|\,x\in\Omega\}<\lambda_{1}\left(\Omega\right)$.
\end{proof}
\begin{rem}\label{rem1}
i)~~For each $h\in L^{2}\left(\Omega\right)$, the problem
\begin{align*}
\left\{\begin{array}{l l}
-\Delta u=h &\mathrm{i}\mathrm{n}\ \Omega,\\
u=0 &\mathrm{on}\ \partial\Omega\\
\end{array}\right.
\end{align*}
has a unique solution $u\in H^{2}(\Omega)$, such as when $\Omega$ is a bounded convex domain with a piecewise $C^{2}$ boundary
 $($see, e.g., {\rm \cite[Section 3.3]{grisvard2011elliptic}}$)$.
Therefore, the so-called bootstrap argument ensures that a weak solution $u\in H_{0}^{1}\left(\Omega\right)$ to \eqref{abs/problem} on such a domain $\Omega$ is in $C^{\infty}\left(\Omega\right)(\subset C^{2}\left(\Omega\right))$.

ii)~~The first eigenvalue of the problem $(\ref{weak/eig/pro})$ can be numerically estimated by, e.g., the method in {\rm \cite{liu2013verified}}, which allows us to concretely evaluate the eigenvalues of $(\ref{weak/eig/pro})$ on polygonal domains.
\end{rem}

\section{Method for estimating the embedding constant}\label{sec/embedding}
In this section, we propose a method for estimating the embedding constant $C_{p+1}\left(\Omega\right)$ defined in $(\ref{bestconstant})$.
Hereafter, $\overline{B}\left(x,r;\ \|\cdot\|\right)$ denotes the closed ball whose center is $x$ and whose radius is $r\geq 0$ in the sense of the norm $\|\cdot\|$.
The following theorem provides an explicit estimation of the embedding constant from a verified solution to $(\ref{positive/problem})$.

\begin{theo}\label{theo/embedding}
Let $\Omega\subset \mathbb{R}^{2}$ be a bounded convex domain.
If there exists a solution to $(\ref{positive/problem})$ in a closed ball $\overline{B}(\hat{u},r\ ;\ \|\cdot\|_{H_{0}^{1}(\Omega)})$ with $\hat{u}\in H_{0}^{1}\left(\Omega\right)$ satisfying $\left\|\hat{u}\right\|_{H_{0}^{1}\left(\Omega\right)}>2r$, then the embedding constant $C_{p+1}\left(\Omega\right)$ defined in $(\ref{bestconstant})$ is estimated as
\begin{align*}
\displaystyle \frac{\left\|\hat{u}\right\|_{L^{p+1}\left(\Omega\right)}}{\left\|\hat{u}\right\|_{H_{0}^{1}\left(\Omega\right)}}\leq C_{p+1}\left(\Omega\right)\leq\frac{\left\|\hat{u}\right\|_{L^{p+1}\left(\Omega\right)}}{\left\|\hat{u}\right\|_{H_{0}^{1}\left(\Omega\right)}-2r}.
\end{align*}
\end{theo}

\begin{proof}
It is clear that $\left\|\hat{u}\right\|_{L^{p+1}\left(\Omega\right)}/\left\|\hat{u}\right\|_{H_{0}^{1}\left(\Omega\right)}$ is a lower bound of $C_{p+1}\left(\Omega\right)$.
A solution to $(\ref{positive/problem})$ is unique when $\Omega\subset \mathbb{R}^{2}$ is bounded and convex \cite{lin1994uniqueness}.
Therefore, the ratio $\left\|u\right\|_{L^{p+1}(\Omega)}/\left\|u\right\|_{H_{0}^{1}\left(\Omega\right)}$ is maximized by the solution $u$ to $(\ref{positive/problem})$.
By writing the solution to \eqref{positive/problem} as $\hat{u}+rv$ with $v\in H_{0}^{1}\left(\Omega\right),\ \left\|v\right\|_{H_{0}^{1}\left(\Omega\right)}\leq 1$, we have that
\begin{align*}
C_{p+1}\displaystyle \left(\Omega\right)=\frac{\left\|\hat{u}+rv\right\|_{L^{p+1}\left(\Omega\right)}}{\left\|\hat{u}+rv\right\|_{H_{0}^{1}\left(\Omega\right)}}\leq\frac{\left\|\hat{u}\right\|_{L^{p+1}\left(\Omega\right)}+rC_{p+1}\left(\Omega\right)}{\left\|\hat{u}\right\|_{H_{0}^{1}\left(\Omega\right)}-r}.
\end{align*}
In other words, it follows that
\begin{align*}
\left(\left\|\hat{u}\right\|_{H_{0}^{1}\left(\Omega\right)}-2r\right)C_{p+1}\left(\Omega\right)\leq\left\|\hat{u}\right\|_{L^{p+1}\left(\Omega\right)}.
\end{align*}
Hence, when $\left\|\hat{u}\right\|_{H_{0}^{1}\left(\Omega\right)}>2r,\ \left\|\hat{u}\right\|_{L^{p+1}\left(\Omega\right)}/(\left\|\hat{u}\right\|_{H_{0}^{1}\left(\Omega\right)}-2r)$ becomes an upper bound of $C_{p+1}\left(\Omega\right)$.
\end{proof}

\begin{rem}
Theorem $\ref{theo/embedding}$ can be naturally applied to the case of $n\geq 3$ under suitable conditions.
For example, if $\Omega$ is a convex symmetric domain in $\mathbb{R}^{n}\ (n\geq 3)$ and $1<p<(n+2)/(n-2)$, \eqref{positive/problem} has only one solution {\rm \cite{grossi2000uniqueness, pacella2005uniqueness}}.
\end{rem}
\section{Numerical example}\label{sec/ex}
In this section, we present some numerical examples where the best values of the embedding constants on the square domain $\Omega_{s}=\left(0,1\right)^{2}$ are estimated to yield Theorem \ref{prop/embedding}.
The upper bounds for the embedding constants on the L-shaped domain $\Omega_{L}=(0,2)^{2} \backslash [1,2]^{2}$ through the application of Theorem \ref{prop/embedding} are also presented.
All computations were carried out on a computer with Intel Xeon E7-4830 2.20 GHz$\times$40 processors, 2 TB RAM, CentOS 6.6, and MATLAB 2012b.
All rounding errors were strictly estimated by using toolboxes for the verified numerical computations: the INTLAB version 9 \cite{rump1999book} and KV library version 0.4.16 \cite{kashiwagikv}.
Therefore, the accuracy of all results was guaranteed mathematically.

We consider the cases where $p=2,\,3,\,4,\,5,\,{\rm and}~6$, which correspond to the critical point problems for embedding constants $C_{p+1}(\Omega)$.
We computed approximate solutions $\hat{u}$ to \eqref{abs/problem}, which are displayed in Fig.~\ref{pic}, with Legendre polynomials, i.e., we constructed $\hat{u}$ as
\begin{align}
\displaystyle \hat{u}=\sum_{i,j=1}^{N}u_{i,j}\phi_{i}\phi_{j},~~u_{i,j}\in \mathbb{R},
\end{align}
where each $\phi_{i}$ is defined by
\begin{align}
\displaystyle \phi_{n}(x)=\frac{1}{n(n+1)}x(1-x)\frac{dP_{n}}{dx}(x),~~n=1,2,3,\cdots
\end{align}
with the Legendre polynomials $P_{n}$ defined by
\begin{align}
P_{n}=\displaystyle \frac{(-1)^{n}}{n!}\left(\frac{d}{dx}\right)^{n}x^{n}(1-x)^{n},~~n=0,1,2,\cdots.
\end{align}
We used \cite[Theorem 2.3]{kimura1999on} to obtain a concrete interpolation error bound for the basis $\left\{\phi_{i}\right\}_{i=1}^{N}$.
Other important properties of the basis are also discussed in \cite{kinoshita2015recurrence}.
We then proved the existence of solutions $u$ to \eqref{abs/problem} in an $H_{0}^{1}$-ball $\overline{B}(\hat{u},r_{1};\ \|\cdot\|_{H_{0}^{1}\left(\Omega_{s}\right)})$ and an $L^{\infty}$-ball $\overline{B}(\hat{u},r_{2};\ \|\cdot\|_{L^{\infty}\left(\Omega_{s}\right)})$, both centered around the approximations $\hat{u}$; this was done by combining the methods described in \cite{plum2001computer} and \cite{tanaka2014verified}.
On the basis of the ``Existence and Enclosure Theorem'' on p.~154 in \cite{plum2001computer}, we obtained the $H_{0}^{1}$-ball.
The required constants $\delta$ and $K$ and function $G$ in the theorem were computed as follows:
\begin{itemize}
\setlength{\parskip}{1pt}
\setlength{\itemsep}{1pt}
\item $\delta$ was computed as $\delta=C_{2}(\Omega_{s})\|\Delta\hat{u}+\hat{u}^{p}\|_{L^{2}(\Omega_{s})}$;
\item $K$ was computed by the method described in \cite{tanaka2014verified};
\item The method described in \cite[Section~3.4]{plum2009computer} was used for the explicit selection of $G$.
\end{itemize}
By using \cite[Theorem 4.1]{plum2001computer}, we also obtained the $L^{\infty}$-ball.
Note that the verified solution has the regularity to be in $C^{2}\left(\Omega_{s}\right)\cap C\left(\overline{\Omega_{s}}\right)$ regardless of the regularity of the approximations $\hat{u}$ owing to the argument in Remark \ref{rem1} i).
Table \ref{veri/result} presents the verification results, which ensure the positiveness of the verified solutions to \eqref{abs/problem} centered around $\hat{u}$ for cases $p=3~{\rm and}~5$ owing to the condition $\displaystyle \sup\{(u_{-}\left(x\right))^{p-1}\,|\, x\in\Omega_{s}\}<\lambda_{1}(=2\pi)$.
Here, the upper bounds of $\displaystyle \sup\{(u_{-}\left(x\right))^{p-1}\,|\, x\in\Omega_{s}\}$ were calculated by $\left(\left|\min\{\hat{u}\left(x\right)\ |\ x\in\Omega_{s}\}\right|+r_{2}\right)^{p-1}$ with verification.
The last column in the table presents intervals containing $C_{p+1}\left(\Omega_{s}\right)$, e.g., $1.23_{456}^{789}$ represents the interval [1.23456,1.23789].
These intervals yield the results in Theorem \ref{prop/embedding}.
Table \ref{compare} compares the lower and upper bounds derived by our method, the upper bounds derived by Corollary \ref{roughboundtheo}, and the upper bounds derived by Plum's formula \cite{plum2008} (Theorem \ref{plumembedding}).

In addition, we applied the results of Theorem \ref{prop/embedding} to estimate the upper bounds of the embedding constants on $\Omega_{L}$.
Since $\Omega_{L} \subset (0,2)^2$, which is the smallest square that encloses $\Omega_{L}$, $C_{p}(\Omega_{L})$ is bounded by $2^{2/p}C_{p}(\Omega_{s})$ owing to the discussion in Remark \ref{remark1}.
Table \ref{compareL} compares $2^{2/p}C_{p}(\Omega_{s})$ derived by our method, the upper bounds for $C_{p}(\Omega_{L})$ derived by Corollary \ref{roughboundtheo}, and the upper bounds for $C_{p}(\Omega_{L})$ derived by Theorem \ref{plumembedding}.
Theorem \ref{plumembedding} requires a concrete value for the minimum eigenvalue of $-\Delta$.
Therefore, we employed the result of $\lambda_{1} \geq 9.5585$ presented in \cite[Table 5.1]{liu2014high}.

\begin{table}[!h]
\caption{Verification results for the cases $p=2,\,3,\,4,\,5,\,{\rm and}~6$ on $\Omega_{s}=(0,1)^{2}$.}
\label{veri/result}
\begin{center}
 \renewcommand\arraystretch{1.3}
 \footnotesize
 \begin{tabular}{c|ccccl}
 \hline
 $p$&
	$N$&
 $r_{1}$&
 $r_{2}$&
 $\displaystyle \sup\,(u_{-}(x))^{p-1}$&
 \multicolumn{1}{c}{$C_{p+1}\left(\Omega_{s}\right)$}\\
 \hline
 \hline
 2&
	100&
	1.289071378e-12&
	3.643032347e-12& 
	-&
 $0.257124750176_{18}^{20}$\\
 3&
	150&
	6.636469152e-13&
	4.363745213e-12& 
	1.904227228e-23& 
 $0.2852444607192_{5}^{9}$\\
 4&
	150&
	5.841283013e-13&
	2.002871834e-11& 
	-&
 $0.310580150945_{05}^{12}$\\
 5&
	150&
	7.088374332e-13&
	1.724519836e-10& 
	8.844489601e-40& 
 $0.333840421511_{02}^{12}$\\
 6&
	200&
	1.310581865e-12&
	4.769644376e-08& 
	-&
 $0.355479942886_{11}^{34}$\\
 \hline
 \end{tabular}
\end{center}
\end{table}

\begin{table}[h]
\caption{Estimates of $C_{p}\left(\Omega_{s}\right)$ derived by our method, Corollary \ref{roughboundtheo}, and Theorem \ref{plumembedding}.}
\label{compare}
\begin{center}
 \renewcommand\arraystretch{1.3}
 \footnotesize
 \begin{tabular}{c|lcc}
 \hline
 $C_{p}\left(\Omega_{s}\right)$&
 \multicolumn{1}{c}{Our method}&
 Corollary \ref{roughboundtheo}&
 Theorem \ref{plumembedding}\\
 \hline
 \hline
 $C_{3}\left(\Omega_{s}\right)$&
 $0.257124750176_{18}^{20}$&
 $0.27991104681668$&
 $0.32964899322075$\\
 $C_{4}\left(\Omega_{s}\right)$&
 $0.2852444607192_{5}^{9}$&
 $0.31830988618380$&
 $0.39894228040144$\\
 $C_{5}\left(\Omega_{s}\right)$&
 $0.310580150945_{05}^{12}$&
 $0.35780388458051$&
 $0.48909030972535$\\
 $C_{6}\left(\Omega_{s}\right)$&
 $0.333840421511_{02}^{12}$&
 $0.39585399866620$&
 $0.55266945714001$\\
 $C_{7}\left(\Omega_{s}\right)$&
 $0.355479942886_{11}^{34}$&
 $0.43211185419351$&
	$0.63763213907292$\\
 \hline
 \end{tabular}
\end{center}
\end{table}

\begin{table}[!h]
\caption{Same as Table \ref{compare} but for $\Omega_{L}=(0,2)^{2} \backslash [1,2]^{2}$.}
\label{compareL}
\begin{center}
 \renewcommand\arraystretch{1.3}
 \footnotesize
 \begin{tabular}{c|ccc}
 \hline
 $C_{p}\left(\Omega_{L}\right)$&
 \multicolumn{1}{c}{Our method ($2^{2/p}C_{p}(\Omega_{s}))$}&
 Corollary \ref{roughboundtheo}&
 Theorem \ref{plumembedding}\\
 \hline
 \hline
 $C_{3}\left(\Omega_{L}\right)$&
 $0.51424950035240$&
 $0.40370158699565$&
 $0.41978967493887$\\
 $C_{4}\left(\Omega_{L}\right)$&
 $0.45279735701391$&
 $0.41891936927236$&
 $0.47823908300428$\\
 $C_{5}\left(\Omega_{L}\right)$&
 $0.43922666167046$&
 $0.44572736933656$&
 $0.56542767015609$\\
 $C_{6}\left(\Omega_{L}\right)$&
 $0.44050507711968$&
 $0.47539569585243$&
 $0.62367087563741$\\
 $C_{7}\left(\Omega_{L}\right)$&
 $0.44787666285793$&
 $0.50554097277928$&
 $0.70723155088841$\\
 \hline
 \end{tabular}
\end{center}
\end{table}

\newcommand{\sizee}{0.5\hsize}
\newcommand{\vminus}{\vspace{-3mm}}
\begin{figure}[p]
\begin{minipage}{\sizee}
	\begin{center}
	\vminus\includegraphics[height=50 mm]{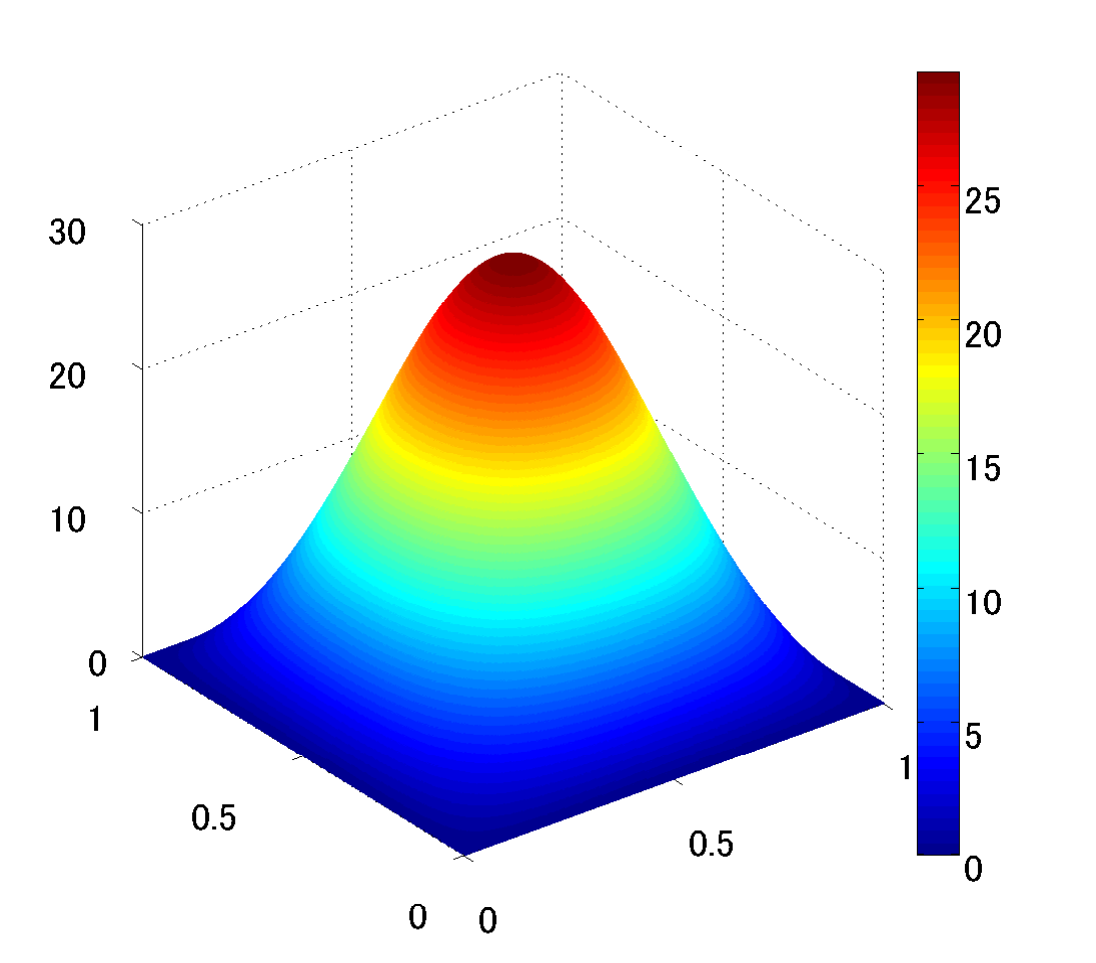}\\
	\footnotesize$p=2$,~~$\displaystyle \max_{x\in\Omega_{s}}\hat{u}(x)\approx 29.2571$
	\end{center}
	~
\end{minipage}
\begin{minipage}{\sizee}
	\begin{center}
	 \vminus\includegraphics[height=50 mm]{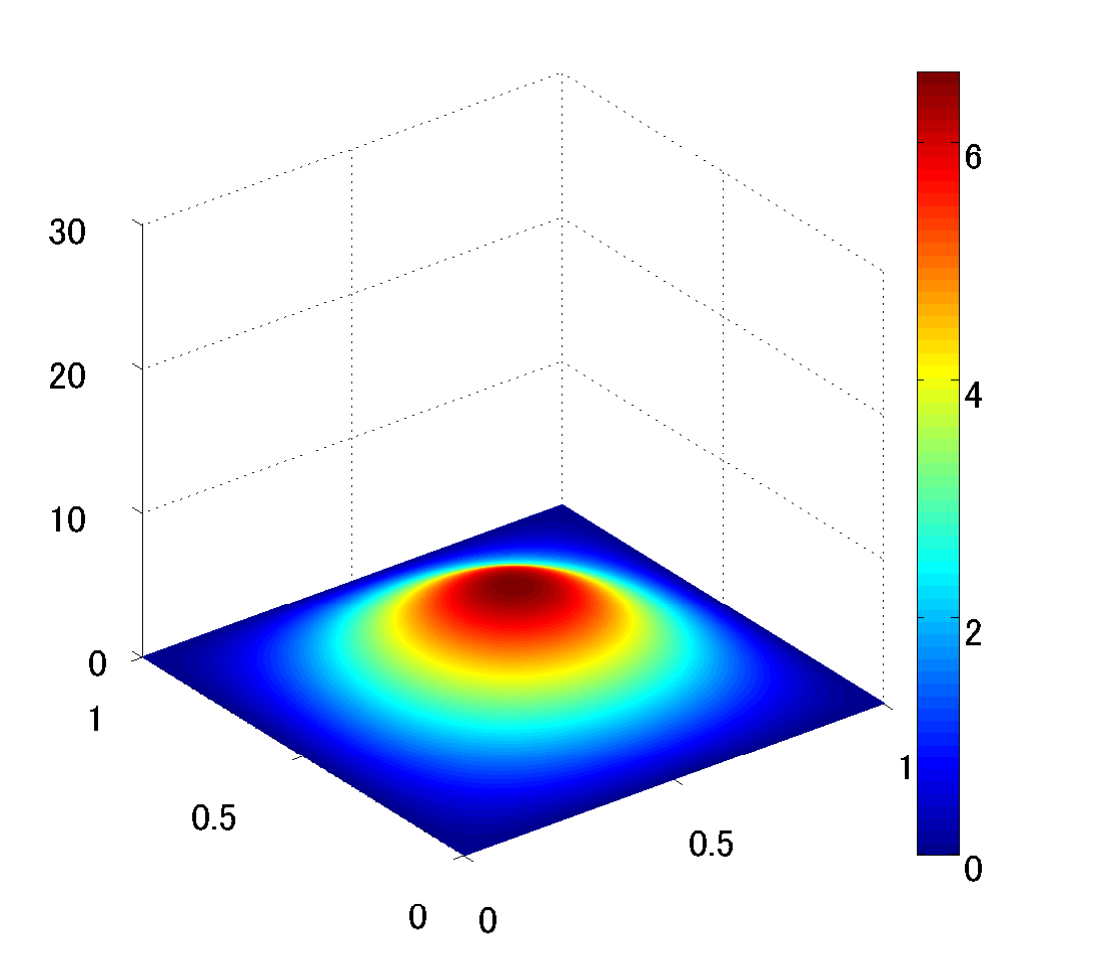}\\
	 \footnotesize$p=3$,~~$\displaystyle \max_{x\in\Omega_{s}}\hat{u}(x)\approx 6.6232$
	\end{center}
	~
\end{minipage}
\begin{minipage}{\sizee}
	\begin{center}
	 \vminus\includegraphics[height=50 mm]{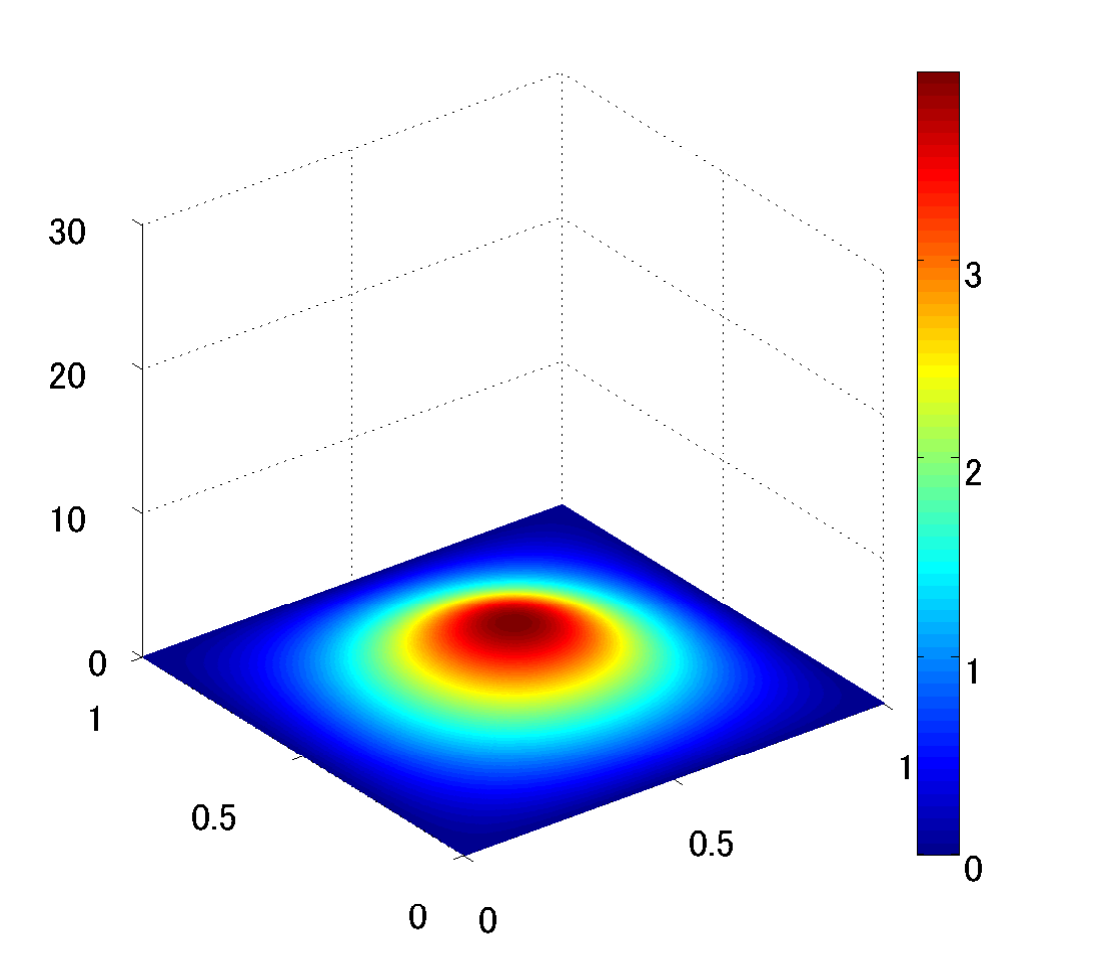}\\
	 \footnotesize$p=4$,~~$\displaystyle \max_{x\in\Omega_{s}}\hat{u}(x)\approx 4.0491$
	\end{center}
	~
\end{minipage}
\begin{minipage}{\sizee}
	\begin{center}
	 \vminus\includegraphics[height=50 mm]{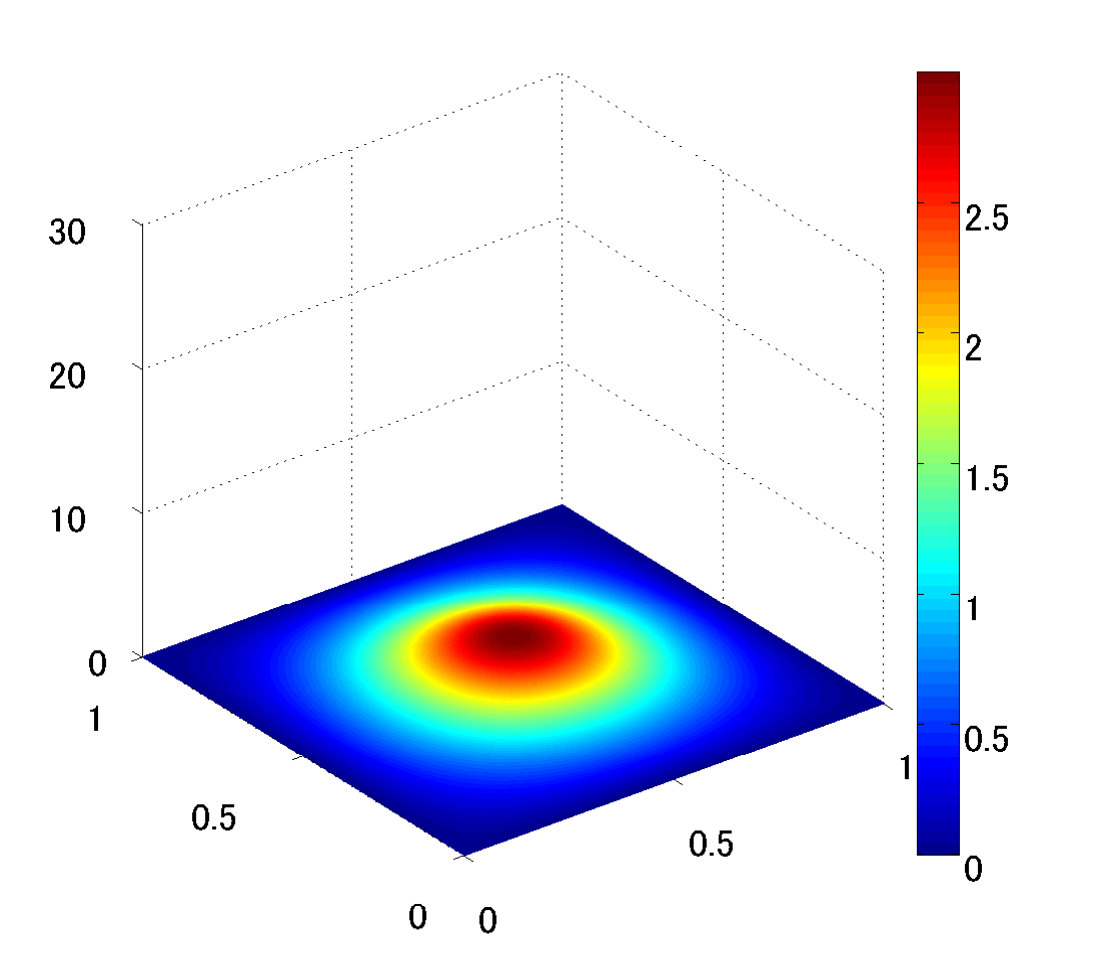}\\
	 \footnotesize$p=5$,~~$\displaystyle \max_{x\in\Omega_{s}}\hat{u}(x)\approx 3.1721$
	\end{center}
	~
\end{minipage}
\begin{minipage}{\sizee}
	\begin{center}
	 \vminus\includegraphics[height=50 mm]{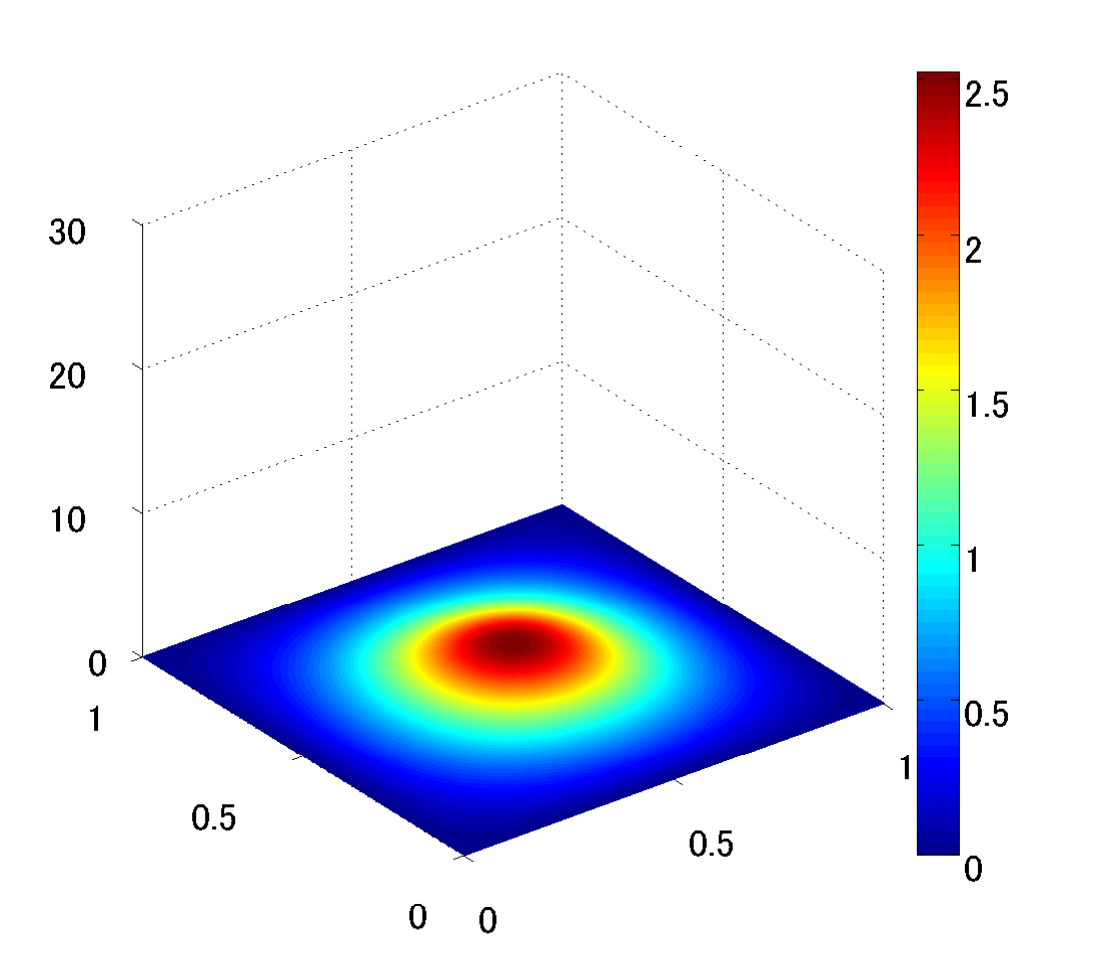}\\
	 \footnotesize$p=6$,~~$\displaystyle \max_{x\in\Omega_{s}}\hat{u}(x)\approx 2.7435$
	\end{center}
	~
\end{minipage}
\caption{Approximate solutions to $(\ref{abs/problem})$ in $\Omega_{s}$ for $p=2,\,3,\,4,\,5,\,{\rm and}~6$.}
\label{pic}
\end{figure}
\section{Conclusion}
In this paper, we proposed a numerical method for obtaining a sharp inclusion of the best constant $C_{p}\left(\Omega\right)$ that satisfies \eqref{embedding} for a bounded convex domain $\Omega$.
We obtained inclusions of $C_{p}\left(\Omega\right)$ on the basis of Theorem \ref{theo/embedding}.
Since Theorem \ref{theo/embedding} requires a concrete inclusion of the extremal function corresponding to target $C_{p}\left(\Omega\right)$, we obtained an inclusion of the function by computing a solution to problem \eqref{positive/problem} with verification.
In order to verify the solutions to \eqref{positive/problem}, we first verified solutions to \eqref{abs/problem} and then proved their positiveness by using Theorem \ref{positive/theo}, except for the case where the positiveness of the solutions was ensured a priori.
The accuracy of all results, e.g., those in Theorem \ref{prop/embedding}, was mathematically guaranteed by using toolboxes for verified numerical computations \cite{rump1999book,kashiwagikv}.
In future work, we would like to present numerical inclusions for the best value of the embedding constant for other cases, e.g., when $p$ is a fraction or $\Omega$ is a more complicated polygonal domain.
\appendix
\section{Simple bounds for the embedding constant}
The following theorem provides the best constant in the classical Sobolev inequality.

\begin{theoA}[T.~Aubin \cite{aubin1976} and G.~Talenti \cite{talenti1976}]\label{talentitheo}
Let $u$ be any function in $W^{1,q}\left(\mathbb{R}^{n}\right)\ (n=2,3,\cdots)$.
Moreover, let $q$ be any real number such that $1<q<n$, and let $p=nq/\left(n-q\right)$.
Then,
\begin{align*}
\left(\int_{\mathbb{R}^{n}}\left|u(x)\right|^{p}dx\right)^{\frac{1}{p}}\leq T_{p}\left(\int_{\mathbb{R}^{n}}\left|\nabla u(x)\right|_{2}^{q}dx\right)^{\frac{1}{q}}
\end{align*}
holds for
\begin{align}
T_{p}=\pi^{-\frac{1}{2}}n^{-\frac{1}{q}}\left(\frac{q-1}{n-q}\right)^{1-\frac{1}{q}}\left\{\frac{\Gamma\left(1+\frac{n}{2}\right)\Gamma\left(n\right)}{\Gamma\left(\frac{n}{q}\right)\Gamma\left(1+n-\frac{n}{q}\right)}\right\}^{\frac{1}{n}}\label{talenticonst},
\end{align}
where
$\left|\nabla u\right|_{2}=\left((\partial u/\partial x_{1})^{2}+(\partial u/\partial x_{2})^{2}+\cdots+(\partial u/\partial x_{n})^{2}\right)^{1/2}$,
and $\Gamma$ denotes the Gamma function.
\end{theoA}
The following corollary obtained from Theorem \ref{talentitheo} provides a simple bound for the embedding constant from $H_{0}^{1}\left(\Omega\right)$ to $L^{p}(\Omega)$ for a bounded domain $\Omega$.

\begin{coroA}\label{roughboundtheo}
Let $\Omega\subset \mathbb{R}^{n} (n=2,3,\cdots)$ be a bounded domain.
Let $p$ be a real number such that $p\in(n/(n-1),2n/(n-2)]$ if $n\geq 3$ and $p\in(n/(n-1),\infty)$ if $n=2$.
In addition, set $q=np/(n+p)$.
Then, $(\ref{embedding})$ holds for
\begin{align*}
	C_{p}\left(\Omega\right)=\left|\Omega\right|^{\frac{2-q}{2q}}T_{p},
\end{align*}
where $T_{p}$ is the constant in {\rm (\ref{talenticonst})}.
\end{coroA}

\begin{proof}
By zero extension outside $\Omega$, we may regard $u\in H_{0}^{1}\left(\Omega\right)$ as an element $u\in H_{0}^{1}\left(\mathbb{R}^{n}\right)$.
Therefore, from Theorem \ref{talentitheo},
\begin{align}
	\left\|u\right\|_{L^{p}\left(\Omega\right)}
	\leq T_{p}\left(\int_{\Omega}\left|\nabla u\left(x\right)\right|_{2}^{q}dx\right)^{\frac{1}{q}}.\label{embedding/theo/1}
\end{align}
H\"{o}lder's inequality gives
\begin{align}
\int_{\Omega}\left|\nabla u\left(x\right)\right|_{2}^{q}dx&
\leq\left(\int_{\Omega}
\left|
	\nabla u\left(x\right)
\right|_{2}^{q\cdot\frac{2}{q}}dx\right)^{\frac{q}{2}}\left(\int_{\Omega}1^{\frac{2}{2-q}}dx\right)^{\frac{2-q}{2}}\nonumber\\
&=\left|\Omega\right|^{\frac{2-q}{2}}
	\left(\int_{\Omega}\left|
	\nabla u\left(x\right)
\right|_{2}^{2}dx\right)^{\frac{q}{2}},\nonumber
\end{align}
that is,
\begin{align}
\left(\int_{\mathbb{R}^{n}}\left|\nabla u\left(x\right)\right|_{2}^{q}dx\right)^{\frac{1}{q}}
\leq\left|\Omega\right|^{\frac{2-q}{2q}}\left\|\nabla u\right\|_{L^{2}\left(\Omega\right)},\label{embedding/theo/2}
\end{align}
where $\left|\Omega\right|$ is the measure of $\Omega$. From (\ref{embedding/theo/1}) and (\ref{embedding/theo/2}), it follows that

\begin{align*}
\left\|u\right\|_{L^{p}\left(\Omega\right)}&
\leq\left|\Omega\right|^{\frac{2-q}{2q}}T_{p}\left\|\nabla u\right\|_{L^{2}\left(\Omega\right)}.
\end{align*}
\end{proof}

Using the following theorem, an upper bound of the embedding constant can be obtained when the minimal point of the spectrum of $-\Delta$ on $H_{0}^{1}(\Omega)$ is explicitly estimated, where $H_{0}^{1}(\Omega)$ is endowed with the inner product
\begin{align}
(\cdot,\cdot)_{H_{0}^{1}(\Omega)}:=(\nabla\cdot,\nabla\cdot)_{L^{2}(\Omega)}+\sigma(\cdot,\cdot)_{L^{2}(\Omega)}\label{sigmanorm}
\end{align}
with a nonnegative number $\sigma$.
Recall that we selected $\sigma =0$ when computing the numerical values in the last columns of Tables \ref{compare} and \ref{compareL}.
\begin{theoA}[M.~Plum \cite{plum2008}]\label{plumembedding}
Let $\rho\in[0,\infty)$ denote the minimal point of the spectrum of $-\Delta$ on $H_{0}^{1}(\Omega)$ endowed with the inner product \eqref{sigmanorm},
where $\sigma$ is selected so that $\sigma>0$ when $\rho=0$.

\noindent$a)$~~Let $n=2$ and $p\in[2,\infty).$
With the largest integer $\nu$ satisfying $\nu\leq p/2,\ (\ref{embedding})$ holds for
\begin{align*}
C_{p}\left(\Omega\right)=\left(\frac{1}{2}\right)^{\frac{1}{2}+\frac{2\nu-3}{p}}\left[\frac{p}{2}\left(\frac{p}{2}-1\right)\cdots\left(\frac{p}{2}-\nu+2\right)\right]^{\frac{2}{p}}\left(\rho+\frac{p}{2}\sigma\right)^{-\frac{1}{p}},
\end{align*}
where $\displaystyle \frac{p}{2}\left(\frac{p}{2}-1\right)\cdots\left(\frac{p}{2}-\nu+2\right)=1$ if $\nu=1.$\\[1pt]
$b)$~~Let $n\geq 3$ and $p\in[2,2n/(n-2)]$.
With $s:=n(p^{-1}-2^{-1}+n^{-1})\in[0,1],\ (\ref{embedding})$ holds for
\begin{align*}
C_{p}\left(\Omega\right)=\left(\frac{n-1}{\sqrt{n}\left(n-2\right)}\right)^{1-s}\left(\frac{s}{s\rho+\sigma}\right)^{\frac{s}{2}}.
\end{align*}
\end{theoA}
\section*{Acknowledgments}
The authors express their sincere thanks to Professor Y.~Watanabe (Kyushu University, Fukuoka, Japan) for referring them to useful literature on previously known results for the basis constructed by Legendre polynomials.
They also express their profound gratitude to an editor and two anonymous referees for their highly insightful comments and suggestions.

  The first author (K.T.) is supported by 
  the Waseda Research Institute for Science and Engineering, the Grant-in-Aid for Young Scientists (Early Bird Program).
The second author (K.S.) is supported by JSPS KAKENHI Grant Number 16K17651.
\bibliographystyle{elsarticle-num}
\bibliography{ref.bib}

\end{document}